\documentclass[11pt]{article}
\usepackage{amsfonts}
\usepackage{graphics}
\usepackage{indentfirst}
\usepackage{color}
\usepackage{cite}
\usepackage{latexsym}
\usepackage[paper=a4paper, left=1.6cm, right=1.6cm, top=1.8cm, bottom=1.6cm, headheight=5.5pt, footskip=0.8cm, footnotesep=0.8cm, centering, includefoot]{geometry}
\usepackage{amsmath}
\allowdisplaybreaks
\usepackage{amssymb}
\usepackage[colorlinks, linkcolor=red]{hyperref}
\hypersetup{urlcolor=red, citecolor=red}
\usepackage[dvips]{epsfig}
\usepackage{amscd}

\usepackage{amsthm}
\usepackage{mathrsfs}
\usepackage{verbatim}
\newtheorem{theorem}{Theorem}[section]
\newtheorem{remark}{Remark}[section]
\newtheorem{defn}{Definition}[section]
\newtheorem{lemma}{Lemma}[section]

\allowdisplaybreaks
\let\p=\perp
\let\pa=\partial
\let\f=\frac

\let\D=\Delta
\let\r=\rho
\let\vr=\varrho

\def\Om{\Omega}
\def\na{\nabla}

\def\cb{{\mathcal B}}

\def\D{{\mathcal D}}
\def\E{{\mathcal E}}

\def\cm{{\mathcal M}}

\renewcommand{\div}{{\rm div}}

\DeclareMathOperator{\divv}{div}
\DeclareMathOperator{\loc}{loc}
\DeclareMathOperator{\curl}{curl}

\makeatletter
\@addtoreset{equation}{section}
\makeatother
\makeatletter
\@addtoreset{equation}{section}
\makeatother

\title{Global stability for compressible isentropic Navier--Stokes equations in 3D bounded domains with Navier-slip boundary conditions
\thanks{Liu's research was partially supported by Natural Science Foundation of Changchun Normal University (No. CSJJ2024003GZR). Wu's research was partially supported by Fujian Alliance of Mathematics (No. 2023SXLMMS08). Zhong's research was partially supported by Fundamental Research Funds for the Central Universities (No. SWU--KU24001) and National Natural Science Foundation of China (No. 12371227).}
}

\author{Yang Liu$\,^{\rm 1}\,$,\ Guochun Wu$\,^{\rm 2}\,$,\
Xin Zhong$\,^{\rm 3}\,$ {\thanks{E-mail addresses:
liuyang0405@ccsfu.edu.cn (Y. Liu), guochunwu@126.com (G. Wu),  xzhong1014@amss.ac.cn (X. Zhong).}}\date{}\\
\footnotesize $^{\rm 1}\,$
College of Mathematics, Research Institute for Scientific and Technological Innovation,
\\ \footnotesize Changchun Normal
University, Changchun 130032, P. R. China\\
	\footnotesize $^{\rm 2}\,$
   School of Mathematics and Statistics, Xiamen University of Technology, Xiamen 361024, P. R. China\\
	\footnotesize $^{\rm 3}\,$ School of Mathematics and Statistics, Southwest University, Chongqing 400715, P. R. China}
\begin{document}
\maketitle

\begin{abstract}
We investigate the global stability of large solutions to the compressible isentropic Navier--Stokes equations in a three-dimensional (3D) bounded domain with Navier-slip boundary conditions. It is shown that the strong solutions converge to an equilibrium state exponentially in the $L^2$-norm provided the density is essentially uniform-in-time bounded from above. Moreover, we obtain that the density converges to its equilibrium state exponentially in the $L^\infty$-norm if additionally the initial density is bounded away from zero. Furthermore, we derive that the vacuum states will not vanish for any time provided vacuum appears (even at a point) initially. This is the first result concerning the global stability for large strong solutions of compressible Navier--Stokes equations with vacuum in 3D general bounded domains.
\end{abstract}

\textit{Key words and phrases}. Compressible Navier--Stokes equations;
global stability; Navier-slip boundary conditions.

2020 \textit{Mathematics Subject Classification}. 76N06; 76N10.


\section{Introduction and main results}
The motion of a viscous compressible isentropic flow in a domain $\Omega\subset\mathbb{R}^3$ is governed by the compressible Navier--Stokes equations
\begin{align}\label{a1}
\begin{cases}
\rho_t+\divv (\rho u)=0,\\
(\rho u)_t+\divv (\rho u\otimes u)+\nabla P=\mu\Delta u+(\mu+\lambda)\nabla\divv  u,
\end{cases}
\end{align}
where the unknowns $\rho$, $u=(u^1, u^2, u^3)$, and $P=a\rho^\gamma\ (a>0,\gamma>1)$ stand for the density, velocity, and pressure, respectively. The constants $\mu$ and $\lambda$ represent the shear viscosity and the bulk viscosity respectively satisfying the physical restrictions
\begin{align*}
\mu>0,\quad 2\mu+3\lambda\ge 0,
\end{align*}
Without loss of generality, we assume that $a=1$ throughout the paper.

Let $\Omega\subset\mathbb{R}^3$ be a simply connected smooth bounded domain,
we consider \eqref{a1} with the initial data
\begin{align}\label{1.2}
\rho(x, 0)=\rho_0(x),\ \rho u(x, 0)=\rho_0u_0(x), \quad x\in \Omega,
\end{align}
and slip boundary conditions
\begin{align}\label{a4}
u\cdot n=0, \ \curl u\times n=0,\quad & x\in\partial\Omega,\ t>0,
\end{align}
where $n$ is the unit outward normal vector to $\pa\Omega$.

Significant progress has been made over the past decades regarding the global existence of solutions to the multi-dimensional compressible isentropic Navier--Stokes equations. Among them, Matsumura and Nishida \cite{MN80,MN83} proved the global well-posedness of classical solutions provided the initial data are close to a non-vacuum equilibrium in $H^3$. Hoff \cite{H95,H952} derived global weak solutions with small and discontinuous data. The major breakthrough on the global solutions with large initial data is due to Lions \cite{L98}, where he employed the renormalization skills to establish global weak solutions to \eqref{a1} in $\mathbb{R}^n$ for $\gamma\geq\frac{3n}{n+2}\ (n=2,3)$. Feireisl--Novotn\'y--Petzeltov\'a \cite{FNP01} later extended Lions' result to the case $\gamma>\frac{n}{2}$ by introducing oscillation defect measure. At the same time, Jiang and Zhang \cite{JZ01,JZ03} obtained global weak solutions for any $\gamma>1$ when the initial data are assumed to have some spherically symmetric or axisymmetric properties. Nonetheless, due to the possible concentration of finite kinetic energy \cite{H21}, it still seems to be a challenge to show the global existence of weak solutions with general 3D data for $\gamma\in (1, \f32]$. Meanwhile, several results are devoted to investigating the global well-posedness of strong (or classical) solutions to the compressible isentropic Navier--Stokes equations with vacuum both at far field and interior region. Huang--Li--Xin \cite{HLX12} established the global existence and uniqueness of classical solutions for the 3D Cauchy problem with smooth initial data which are of small total energy but possibly large oscillations, where the far field density could be either vacuum or non-vacuum. Using some key \textit{a priori} decay with rates (in large time) and a spatially weighted energy method, Li and Xin \cite{LX19} soon after analyzed the 2D Cauchy problem with far-field vacuum provided the initial total energy is suitably small. A central aspect of their analysis in \cite{HLX12,LX19} is the derivation of a time-independent upper bound of the density. Hong--Hou--Peng--Zhu \cite{HHPZ} generalized the main result in \cite{HLX12} to the case of large initial total energy when the adiabatic exponent $\gamma$ is sufficiently close to $1$. Cai and Li \cite{CL23} investigated the global existence of classical solutions in general 3D bounded domains with Navier-slip boundary conditions provided the initial energy is properly small. Let us mention that in the last two decades there were also some interesting results concerning the global well-posedness of compressible isentropic Navier--Stokes equations in Besov spaces, see for instance \cite{CD10,CMZ10,D00,FZZ18,H11,ZLZ20,HHW19}, to cite but a few.

It is also of great interest to study the global stability and large-time behavior of solutions to the compressible Navier--Stokes system. The readers can refer to \cite{MN80,MN83,DX17,DUYZ07,KK02,KK05} for large-time behavior of global solutions with the initial data being a small perturbation of constant equilibrium states, while long-time behavior of Hoff's intermediate weak solutions with discontinuous and small initial data can be found in \cite{HW20,S21}. Naturally, one may wonder how about the large-time behavior of solutions with large initial data. Under the assumptions that the density is essentially bounded and has uniform-in-time positive lower bound, Padula \cite{Padula} proved that weak solutions of \eqref{a1} decay exponentially to the equilibrium state in $L^2$-norm provided the density is essentially bounded from above and below. Such the assumption was later removed in \cite{PSW21} by means of the Bogovskii operator and a suitable Lyapunov functional. Under the hypothesis that the density $\rho(x,t)$ verifies $\inf_{x\in\mathbb R^3}\rho_0(x)\ge c_0>0$ and $\sup_{t\ge0}\|\rho(t)\|_{C^\alpha}\le M$ with arbitrarily small $\alpha$, He--Huang--Wang \cite{HHW19} established the convergence of the large strong solution $(\rho,u)$ to its associated equilibrium state $(1,0)$ with an explicit decay rate being the same as that of the heat equation, more precisely, they obtained
\begin{align*}
\|\rho-1\|_{H^1}+\|u\|_{H^1}\le C(1+t)^{-\f34(\f2p-1)},
\end{align*}
where $(\rho_0-1, u_0)\in L^p(\mathbb R^3)\cap H^2(\mathbb R^3)$ with $p\in [1, 2)$. Shortly thereafter, Gao--Wei--Yao \cite{GWY20} shown that the upper decay rate of the first order spatial derivative converges to zero in $H^1$-norm at the rate $(1+t)^{-\f32(\f1p-\f12)-\f12}$. More recently, Wu and Zhong \cite{WZ24} weakened the assumption $\sup_{t\ge0}\|\rho(t)\|_{C^\alpha}\le M$ in \cite{HHW19} to the {\it bounded density} by employing some new ideas. Meanwhile, Wu--Yao--Zhang \cite{WLZ} investigated the global stability of large strong solutions for \eqref{a1} in the torus $\mathbb{T}^3$ and proved exponential convergence (for large time) in various norms to the reference constant solution. It should be noted that all the results \cite{WLZ,HHW19,GWY20,WZ24} are devoted to the Cauchy problem in $\mathbb R^3$ or $\mathbb T^3$, and harmonic analysis techniques play key roles more or less.

Selecting appropriate boundary conditions for the system \eqref{a1} in a domain with boundaries is a complex task. A common choice is the homogeneous Dirichlet boundary condition $u = 0$ on $\partial \Omega$, known as the no-slip condition, which models a real fluid adhering to a fixed wall $\partial \Omega$, as formulated by Stokes. However, for rough boundaries, the no-slip condition becomes invalid (cf. \cite{serrin1959mathematical}). Navier proposed an alternative, the Navier-type slip condition, given by
$$
u \cdot n=0,\quad (2 D(u) n+\vartheta u)_{t a n}=0 \text { on } \partial \Omega,
$$
where $D(u)=\left(\nabla u+(\nabla u)^{\operatorname{tr}}\right) / 2$ is the shear stress tensor, $\vartheta$ is a scalar friction function, and $v_{\tan}$ denotes the tangential projection of a vector $v$ onto $\partial \Omega$. This condition accounts for a stagnant fluid layer near the wall, allowing slip with a velocity proportional to the shear stress. Such boundary conditions arise from various physical effects, including free capillary boundaries, rough boundaries, perforated boundaries, and exterior electric fields (see, e.g., \cite{APV98,BJ67,B01,JM01}). The Navier-type slip boundary condition is widely utilized in various problems, including free surface flows, turbulence modeling, and inviscid limits (see, e.g., \cite{S82,GL00,B10}).
Then a natural question, but of importance and interest, is that can we show the global stability and large-time behavior of global large solutions to the system \eqref{a1} under {\it Navier slip boundary conditions \eqref{a4}}? The main goal of the present paper is to give an affirmative response to this problem.

Before stating our main result precisely, we describe the notation throughout.
The symbol $\Box$ denotes the end of a proof and $A\triangleq B$ means $A=B$ by definition.
We denote by $\dot{f}\triangleq f_t+u\cdot\nabla f$ the material derivative of $f$ and
\begin{align*}
\int fdx\triangleq\int_\Omega fdx, \quad\bar f\triangleq\frac{1}{|\Omega|}\int_\Omega fdx.
\end{align*}
For $1\le p\le \infty$ and $k\ge 1$, the standard Sobolev spaces are defined as follows
\begin{gather*}
L^p=L^p(\Omega), \ W^{k, p}=W^{k, p}(\Omega), \ H^k=W^{k, 2},\\
\tilde H^1 \triangleq\left\{v\in H^{1}(\Omega)| (v\cdot n)|_{\partial\Omega}=(\curl v\times n)|_{\partial\Omega}=0\right\}.
\end{gather*}
Moreover, we write the initial total energy of \eqref{a1} by
\begin{align}
C_0\triangleq\int\Big(\f12\rho_0|u_0|^2+G(\rho_0)\Big)dx,
\end{align}
with
\begin{align}\label{a7}
G(\rho)\triangleq\rho\int_{\bar\rho}^\rho\f{P(s)-P(\bar\rho)}{s^2}ds.
\end{align}
Now we define precisely what we mean by strong solutions to \eqref{a1}--\eqref{a4}.
\begin{defn}[Strong solutions]
For $T>0$, $(\rho, u)$ is called a strong solution to the problem \eqref{a1}--\eqref{a4} on $\Omega\times [0, T]$ if for some $q\in (3, 6]$,
\begin{align}\label{a8}
\begin{cases}
0\le \rho\in C([0, T]; W^{1, q}),\ \rho_t\in C([0, T]; L^{q}),\\
u\in C([0, T]; H^2)\cap L^2(0, T; W^{2, q}),\\
\sqrt{\rho}u_t\in L^\infty(0, T; L^2), \ u_t\in L^2(0, T; H^1),
\end{cases}
\end{align}
and $(\rho, u)$ satisfies \eqref{a1} a.e. in $\Omega\times[0, T]$.
\end{defn}

Our first main result can be stated as follows.
\begin{theorem}\label{thm1}
Assume that the initial data $(\rho_0\ge 0, u_0)$ satisfies
\begin{align}
K\triangleq\|\rho_0-\bar\rho_0\|_{L^2}
+\|\sqrt{\rho_0}u_0\|_{L^2}+\|\nabla u_0\|_{H^1}<\infty.
\end{align}
Let $(\rho, u)$ be a global strong solution to the problem \eqref{a1}--\eqref{a4} verifying that
\begin{align}\label{a9}
\sup_{t\ge 0}\|\rho(\cdot, t)\|_{L^\infty}\le \hat\rho,
\end{align}
for some positive constant $\hat\rho$. Then there exist some positive constants $C_1$ and $\eta_1$,
which are dependent on $\hat\rho$ and $K$, but independent of $t$, such that
\begin{align}\label{1.8}
\|(\rho-\bar\rho_0,\sqrt{\rho}u, \nabla u,\sqrt{\rho}\dot{u})(\cdot, t)\|_{L^2}\le C_1e^{-\eta_1t}.
\end{align}
If additionally $\inf\limits_{x\in\Omega}\rho_0(x)\ge \rho_*>0$, then
there exist some positive constants $C_2$ and $\eta_2$,
which are dependent on $\rho_*$, $\hat\rho$, and $K$, but independent of $t$, such that
\begin{align}\label{b1.11}
\|\rho-\bar\rho_0\|_{L^\infty}\le C_2e^{-\eta_2t}.
\end{align}
\end{theorem}

\begin{remark}
Theorem \ref{thm1} gives the first global stability result for large strong solutions of compressible Navier--Stokes equations with vacuum in 3D general bounded domains.
\end{remark}

\begin{remark}
It should be noted that Cai and Li \cite{CL23} established the global existence and uniqueness of classical solutions to the problem \eqref{a1}--\eqref{a4} provided the initial energy is small. One of the key ingredients in \cite{CL23} is to derive a time-independent upper bound of the density. Hence, using similar arguments as those in \cite{CL23}, one can show that a strong solution $(\rho, u)$ satisfying \eqref{a9} indeed exists as long as the initial energy is properly small.
\end{remark}

\begin{remark}
The $L^\infty$ stability result \eqref{b1.11} is an improvment of \cite[(1.16)]{CL23}, where the authors showed that there exist positive constants $C$ and $\eta_0$ independent of $t$ such that, for any $t\geq1$ and $r\in[1,\infty)$,
\begin{align*}
\|\rho-\bar\rho_0\|_{L^r}\le Ce^{-\eta_0t}
\end{align*}
provided the initial energy is properly small.
\end{remark}

Next, we obtain that the vacuum states will not vanish for any time provided vacuum appears initially.
\begin{theorem}\label{thm2}
Assume that all conditions of Theorem \ref{thm1} are satisfied. If additionally $\inf\limits_{x\in\Omega}\rho_0(x)=0$, then it holds
that
\begin{align}\label{aa11}
\inf_{x\in\Omega}\rho(x, t)=0,
\end{align}
for any $t\ge 0$.
\end{theorem}

\begin{remark}
\eqref{aa11} shows that the vacuum state will not vanish for any time, which together with \eqref{1.8}, \eqref{b1.11}, and Gagliardo--Nirenberg inequality implies that, for any $r>3$,
\begin{align*}
\|\nabla\rho(\cdot,t)\|_{L^r}\geq \hat{C_1}e^{\hat{C_2}t},
\end{align*}
where $\hat{C_1}$ and $\hat{C_2}$ are positive constants independent of $t$.
\end{remark}

Now we sketch the main ideas used in the proof of Theorems \ref{thm1} and \ref{thm2}. Compared with the previous results \cite{WLZ,WZ24} in which the authors dealt with the Cauchy problem, the Navier-slip boundary condition \eqref{a4} develops additional difficulties in deriving {\it a priori} estimates for solutions of the compressible Navier--Stokes equations. Hence we need some new observations to overcome these obstacles. First, we prove the exponential decay estimate of $\|(\rho-\bar\rho_0, \sqrt\rho u)\|_{L^2}$. On one hand, by making the basic energy estimate and the Bogovskii operator (see Lemma \ref{l22}), one can derive an energy-dissipation inequality of the form
\begin{align}\label{z1}
\f{d}{dt}\tilde\E(t)+\tilde\D(t)\le 0,
\end{align}
where $\tilde\E(t)$ is equivalent to $\|(\rho-\bar\rho_0, \sqrt\rho u)\|_{L^2}^2$ and $\tilde\D(t)$ is equivalent to $\|(\rho-\bar\rho_0, \na u)\|_{L^2}$. On the other hand, by making full use of the momentum equation, Poincar\'e's inequality, and the $L^p$ theory of the div-curl system, one has that
\begin{align*}
\|\na u\|_{L^p}\le C(\|\divv u\|_{L^p}+\|\curl u\|_{L^p}),\ \
\|\sqrt\rho u\|_{L^2}\le C\|\nabla u\|_{L^2},
\end{align*}
which implies that $\tilde\D(t)\ge C\tilde\E(t)$. This combined with \eqref{z1} immediately yields the exponential decay rate of $\|(\rho-\bar\rho_0, \sqrt\rho u)\|_{L^2}$.
Next, we show the exponential decay estimate of $\|\na u(\cdot,t)\|_{L^2}$. To this end, taking good properties of the {\it effective viscous flux} $F$ and the {\it vorticity} $\curl u$, we can obtain an important energy inequality \eqref{lwz3}. Note that $\|\na u(\cdot, t)\|_{L^2}$ is sufficiently small for large enough $t$, which together with \eqref{w15} gives a crucial Lyapunov-type energy inequality \eqref{b37}. This along with Gronwall's inequality implies the desired exponential decay rate of $\|\nabla u(\cdot,t)\|_{L^2}$. Next, in order to obtain the exponential decay estimate of $\|\sqrt{\rho}\dot{u}(\cdot,t)\|_{L^2}$, we need to improve the uniform-in-time bounds of $\|\sqrt{\rho}\dot{u}(\cdot,t)\|_{L^2}$. In this process, we have to deal with the integrals on the boundary $\pa\Omega$, such as $\int_{\pa\Omega}F_t(\dot u\cdot n)dS$ and so on. To overcome this difficulty, inspired by \cite{CL23}, one infers from $(u\cdot n)|_{\partial\Omega}=0$ that
\begin{align}\label{a15}
u\cdot\na u\cdot n=-u\cdot\na n\cdot u,\quad \text{on}~\partial\Omega,
\end{align}
which implies
\begin{align*}
(\dot u+(u\cdot\na n)\times u^\bot)\cdot n=0, \quad \text{on}~\partial\Omega,
\end{align*}
with $u^\bot\triangleq-u\times n$, and the key point is to control the term $\int_{\pa\Om}(u\cdot\na n\cdot u)FdS$ (see \eqref{w51}). Combining these key facts with lower-order energy estimates, we obtain a Lyapunov-type energy inequality \eqref{w58}. This will lead to the exponential decay rate of $\|\sqrt{\rho}\dot{u}(\cdot,t)\|_{L^2}$ immediately.
To prove \eqref{b1.11}, the key ingredient is to get a time-independent
positive lower bound of the density $\rho$. This is the case by modifying the method used in \cite{WZ24}. Then we can prove that $\|\rho-\bar\rho_0\|_{L^\infty}$ decays exponentially as well by the damping mechanism of density (see Lemma \ref{l34}). As a by-product, we finally show that the vacuum states will not vanish for any time provided vacuum appears initially.

The rest of this paper is organized as follows. Some important inequalities and auxiliary lemmas will be given in Section \ref{sec2}. Section \ref{sec3} is devoted to proving Theorems \ref{thm1} and \ref{thm2}.

\section{Preliminaries}\label{sec2}

In this section we recall some known facts and elementary inequalities which will be used later. We start with the following interpolation inequalities (cf. \cite[Theorems 5.8 and 5.9, pp. 139--141]{AF03}).

\begin{lemma}\label{lwz}
Assume that $\Omega$ is a domain in $\mathbb{R}^d$ satisfying the cone condition.
\begin{itemize}
\item [$\bullet$] Let $mr>d$, let $r\leq p\leq\infty$; if $mr=d$, let $r\leq p<\infty$; if $mr<d$, let $r\leq p\leq\frac{dr}{d-mr}$. Then there exists a constant $K$ depending on $m,d,r,p$, and the dimensions of the cone providing the cone condition for $\Omega$ such that, for all $f\in W^{m,r}(\Omega)$,
\begin{align}\label{LWZ}
\|f\|_{L^p}\le K\|f\|_{W^{m,r}}^\theta\|f\|_{L^r}^{1-\theta},
\end{align}
where $\theta=\frac{d}{m}\big(\frac1r-\frac1p\big)$.
\item [$\bullet$] Let $r>1$ and $mr>d$. Suppose that either $1\leq q\leq r$ or both $q>r$ and $mr-r<d$. Then there exists a constant $K$ depending on $m,d,r,q$, and the dimensions of the cone providing the cone condition for $\Omega$ such that, for all $f\in W^{m,r}(\Omega)$,
\begin{align}\label{LWZ2}
\|f\|_{L^\infty}\le K\|f\|_{W^{m,r}}^\theta\|f\|_{L^q}^{1-\theta},
\end{align}
where $\theta=\frac{dr}{dr+q(mr-d)}$.
\end{itemize}
\end{lemma}

Next, since our solution $u$ does not vanish on the boundary,
we shall use the following generalized Poincar{\'e} inequality (cf. \cite[Lemma 8]{BS2012}).
\begin{lemma}\label{l23}
Let $\Omega\subset\mathbb{R}^3$ be bounded with Lipschitz boundary. Then, for $1<p<\infty$, there exists a positive constant $C$ depending only on $p$ and $\Omega$ such that
\begin{align}\label{z2.5}
\|f\|_{L^p} \leq  C\|\nabla f\|_{L^p},
\end{align}
for each vector field $f\in W^{1,p}(\Omega)$ satisfying either $(f\cdot n)|_{\partial\Omega}=0$ or $(f\times n)|_{\partial\Omega}=0$.
\end{lemma}

The combination of Lemma \ref{lwz} and the Poincar{\'e} inequality yields the following Gagliardo--Nirenberg inequality.
\begin{lemma}\label{l21}
(Gagliardo--Nirenberg inequality, special case).
Assume that $\Omega$ is a bounded Lipschitz domain in $\mathbb{R}^3$. For $p\in [2, 6]$, $q\in (1, \infty)$, and $r\in (3, \infty)$, there exist generic constants $C_i>0\ (i\in\{1,2,3,4\})$ which may depend only on $p$, $q$, $r$, and $\Omega$ such that, for any $f\in H^1$ and $g\in L^q\cap D^{1, r}$,
\begin{gather}
\|f\|_{L^p}\le C_1\|f\|_{L^2}^\f{6-p}{2p}\|\na f\|_{L^2}^\f{3p-6}{2p}+C_2\|f\|_{L^2},\label{f1}\\
\|g\|_{L^\infty}\le C_3\|g\|_{L^q}^{\frac{q(r-3)}{3r+q(r-3)}}\|\na g\|_{L^r}^{\frac{3r}{3r+q(r-3)}}+C_4\|g\|_{L^2}.\label{f2}
\end{gather}
Moreover, if $\int_{\Omega}f(x)dx=0$ or $(f\cdot n)|_{\partial\Omega}=0$ or $(f\times n)|_{\partial\Omega}=0$, we can choose $C_2=0$. Similarly, the constant $C_4=0$ provided $\int_{\Omega}g(x)dx=0$ or $(g\cdot n)|_{\partial\Omega}=0$ or $(g\times n)|_{\partial\Omega}=0$.
\end{lemma}
\begin{proof}
For $f\in H^1$, choosing $m=1,r=2$, and $d=3$ in \eqref{LWZ} (we recall that the Lipschitz condition implies the cone condition, see \cite[p. 84]{AF03}), one has that
\begin{align}\label{lz}
\|f\|_{L^p}\le K\|f\|_{H^1}^{\frac{3p-6}{2p}}\|f\|_{L^2}^{\frac{6-p}{2p}}
\le C_1\|f\|_{L^2}^\f{6-p}{2p}\|\na f\|_{L^2}^\f{3p-6}{2p}+C_2\|f\|_{L^2}.
\end{align}
If $\int_{\Omega}f(x)dx=0$, the classical Poincar{\'e} inequality asserts that
\begin{align*}
\|f\|_{L^2}\leq C\|\nabla f\|_{L^2},
\end{align*}
and the same result also holds when $(f\cdot n)|_{\partial\Omega}=0$ or $(f\times n)|_{\partial\Omega}=0$ due to Lemma \ref{l23}. Thus, the right-hand side of the first inequality in \eqref{lz} can be controlled by $C\|f\|_{L^2}^\f{6-p}{2p}\|\na f\|_{L^2}^\f{3p-6}{2p}$ provided $\int_{\Omega}f(x)dx=0$ or $(f\cdot n)|_{\partial\Omega}=0$ or $(f\times n)|_{\partial\Omega}=0$.

For $g\in L^q\cap D^{1, r}$, choosing $m=1$ and $d=3$ in \eqref{LWZ2}, one has that
\begin{align}\label{lz2}
\|g\|_{L^\infty}
& \le K\|g\|_{L^q}^{\frac{q(r-3)}{3r+q(r-3)}}\|g\|_{W^{1,r}}^{\frac{3r}{3r+q(r-3)}} \notag \\
& \le C\|g\|_{L^q}^{\frac{q(r-3)}{3r+q(r-3)}}\|\nabla g\|_{L^r}^{\frac{3r}{3r+q(r-3)}}+C\|g\|_{L^q}^{\frac{q(r-3)}{3r+q(r-3)}}\| g\|_{L^r}^{\frac{3r}{3r+q(r-3)}}
\notag \\
& \le C\|g\|_{L^q}^{\frac{q(r-3)}{3r+q(r-3)}}\|\nabla g\|_{L^r}^{\frac{3r}{3r+q(r-3)}}+C\|g\|_{L^s},
\end{align}
where $s=\max\{q,r\}>2$. It follows from H{\"o}lder's inequality and Young's inequality that, for any $\varepsilon>0$,
\begin{align}\label{lz3}
\|g\|_{L^s}\le C(\varepsilon)\|g\|_{L^2}+\varepsilon\|g\|_{L^\infty}.
\end{align}
Hence, we obtain \eqref{f2} from \eqref{lz2} and \eqref{lz3} after choosing $\varepsilon$ sufficiently small.
Moreover, it is noted that the classical Poincar{\'e} inequality and Lemma \ref{l23} indicate that
\begin{align*}
\|g\|_{W^{1,r}}\leq C\|\nabla g\|_{L^r},
\end{align*}
if $\int_{\Omega}g(x)dx=0$ or $(g\cdot n)|_{\partial\Omega}=0$ or $(g\times n)|_{\partial\Omega}=0$. So
\begin{align}
\|g\|_{L^\infty}
\le C\|g\|_{L^q}^{\frac{q(r-3)}{3r+q(r-3)}}\|\nabla g\|_{L^r}^{\frac{3r}{3r+q(r-3)}},
\end{align}
that is $C_2=0$ provided $\int_{\Omega}g(x)dx=0$ or $(g\cdot n)|_{\partial\Omega}=0$ or $(g\times n)|_{\partial\Omega}=0$.
\end{proof}

\begin{remark}
If we take $d=2$ in Lemma \ref{lwz}, then we can obtain Gagliardo--Nirenberg inequality in 2D bounded domains.
\begin{itemize}
\item [$\bullet$] Assume that $\Omega$ is a bounded Lipschitz domain in $\mathbb{R}^2$.
For $p\in [2, \infty)$, $q\in(1, \infty)$, and $r\in (2, \infty)$, there exist generic constants $C_i>0\ (i\in\{1,2,3,4\})$ which may depend only on $p$, $q$, $r$, and $\Omega$ such that, for $f\in H^1$ and $g\in L^q$ with $\nabla g\in L^{r}$,
\begin{gather*}
\|f\|_{L^p}\le C_1\|f\|_{L^2}^\frac{2}{p}\|\nabla f\|_{L^2}^{1-\frac{2}{p}}+C_2\|f\|_{L^2},\\
\|g\|_{L^\infty}\le C_3\|g\|_{L^q}^\frac{q(r-2)}{2r+q(r-2)}\|\nabla g\|_{L^r}^\frac{2r}{2r+q(r-2)}+C_4\|g\|_{L^2}.
\end{gather*}
Moreover, if $\int_{\Omega}f(x)dx=0$ or $(f\cdot n)|_{\partial\Omega}=0$, we can choose $C_2=0$. Similarly, the constant $C_4=0$ provided $\int_{\Omega}g(x)dx=0$ or $(g\cdot n)|_{\partial\Omega}=0$.
\end{itemize}
\end{remark}

Next, we introduce the Bogovskii operator in a bounded domain, which plays an important role in controlling $\|P\|_{L^2}^2$. One has the following conclusion (cf. \cite[Lemma 3.17]{NS04}).
\begin{lemma}\label{l22}
Let $\Omega\subset\mathbb{R}^3$ be a bounded Lipschitz domain. Then, for any $p\in(1,\infty)$, there exists a linear operator $\mathcal{B}=\left[\mathcal{B}_{1}, \mathcal{B}_{2}, \mathcal{B}_{3}\right]: L^{p}(\Omega) \rightarrow \big(W_{0}^{1, p}(\Omega)\big)^{3}$ such that
\begin{align*}
\begin{cases}
\operatorname{div} \mathcal{B}[f]=f, & x \in \Omega, \\
\mathcal{B}[f]=0, & x \in \partial \Omega,
\end{cases}
\end{align*}
and
\begin{align*}
\|\nabla\mathcal{B}[f]\|_{L^{p}} \leq C(p,\Omega)\|f\|_{L^{p}}.
\end{align*}
Moreover, if $f=\operatorname{div} g$ with $g\in L^p(\Omega)$ satisfying $(g \cdot n)|_{\partial \Omega}=0$, it holds that
\begin{align*}
\|\mathcal{B}[f]\|_{L^{p}} \leq C(p,\Omega)\|g\|_{L^{p}}.
\end{align*}
\end{lemma}

The following two lemmas are given in \cite[Theorem 3.2]{WW92} and \cite[Propositions 2.6--2.9]{A2014}.
\begin{lemma}\label{lm24}
Let $k\ge 0$ be an integer number and $1<q<+\infty$. Assume that $\Omega\subset\mathbb R^3$ is a simply connected bounded domain with $C^{k+1, 1}$ boundary $\pa\Omega$. Then, for $v\in W^{k+1, q}$ with $(v\cdot n)|_{\partial\Omega}=0$, it holds that
\begin{align}\label{ff5}
\|v\|_{W^{k+1, q}}\le C(\|\div v\|_{W^{k, q}}+\|\curl v\|_{W^{k, q}}).
\end{align}
 In particular, for $k=0$, we have
 \begin{align}
 \|\na v\|_{L^q}\le C(\|\div v\|_{L^q}+\|\curl v\|_{L^q}).
 \end{align}
\end{lemma}

\begin{lemma}\label{lm25}
Let $k\geq0$ be an integer number and $1<q<+\infty$. Suppose that $\Omega\subset\mathbb R^3$ is a bounded domain and its $C^{k+1,1}$ boundary $\partial\Omega$ has a finite number of 2-dimensional connected components. Then, for $v\in W^{k+1,q}$ with $(v\times n)|_{\partial\Omega}=0$, we have
$$\|v\|_{W^{k+1,q}}\leq C(\|\div v\|_{W^{k,q}}+\|\curl v\|_{W^{k,q}}+\|v\|_{L^q}).$$
In particular, if  $\Omega$ has no holes, then
$$\|v\|_{W^{k+1,q}}\leq C(\|\div v\|_{W^{k,q}}+\|\curl v\|_{W^{k,q}}).$$
\end{lemma}

Motivated by \cite{H95,H952}, we set
\begin{align}\label{f5}
F\triangleq(2\mu+\lambda)\divv u-(P-\bar P), \quad \omega\triangleq\na\times u,
\end{align}
where $F$ and $\omega$ denote the {\it effective viscous flux} and the {\it vorticity}, respectively. For $F$, $\omega$, and $\na u$, we have the following key {\it a priori} estimates.
\begin{lemma}\label{l24}
Let $(\rho, u)$ be a classical solution of \eqref{a1}--\eqref{a4} on $\Omega\times(0, T]$. Then, for any $p\in[2, 6]$ and $1<q<+\infty$, there exists a positive constant $C$ depending only on $p$, $q$, $\mu$, $\lambda$, and $\Omega$ such that
\begin{gather}
\|\nabla u\|_{L^q}\le C\big(\|\divv u\|_{L^q}+\|\omega\|_{L^q}\big), \label{f6}
\\
\|\nabla F\|_{L^p}+\|\nabla\omega\|_{L^p}\le C\big(\|\rho\dot u\|_{L^p}+\|\nabla u\|_{L^2}\big),\label{f8}\\
\|F\|_{L^p}\le C\|\rho\dot u\|_{L^2}^{\frac{3p-6}{2p}}\big(\|\nabla u\|_{L^2}+\|P-\bar P\|_{L^2}\big)^{\frac{6-p}{2p}}
+C\big(\|\nabla u\|_{L^2}+\|P-\bar P\|_{L^2}\big),
\\
\|\omega\|_{L^p}\le C\|\rho\dot u\|_{L^2}^{\frac{3p-6}{2p}}\|\nabla u\|_{L^2}^{\frac{6-p}{2p}}+C\|\nabla u\|_{L^2}.
\end{gather}
Moreover, one has that
\begin{align}\label{f12}
\|\nabla u\|_{L^p}&\le C\|\rho\dot u\|_{L^2}^{\frac{3p-6}{2p}}\big(\|\nabla u\|_{L^2}+\|P-\bar P\|_{L^2}\big)^{\frac{6-p}{2p}}
+C\big(\|\nabla u\|_{L^2}+\|P-\bar P\|_{L^p}\big).
\end{align}
\end{lemma}
\begin{proof}
Owing to $(u\cdot n)|_{\pa\Omega}=0$, we get \eqref{f6} from Lemma \ref{lm24}.

Taking $\eqref{a1}_2$ and the slip boundary condition $\eqref{a4}$, one finds that the effective viscous flux $F$ satisfies
\begin{align}
\begin{cases}
\Delta F=\div(\rho\dot u) &\text{in}~\Omega,\\
\f{\pa F}{\pa n}=\rho\dot u\cdot n &\text{on}~\pa\Omega.
\end{cases}
\end{align}
It follows from \cite[Lemma 4.27]{NS04} that
\begin{align}\label{ff14}
\|\na F\|_{L^q}\le C\|\rho\dot u\|_{L^q}.
\end{align}
On the other hand, we can rewrite $\eqref{a1}_2$ as
\begin{align*}
\mu\na\times w=\na F-\rho\dot u,
\end{align*}
which along with $(\omega\times n)|_{\pa\Omega}=0$, $\divv w=0$, and Lemma \ref{lm25} leads to
\begin{align}\label{ff16}
\|\na\omega\|_{L^p}\le C(\|\na\times\omega\|_{L^q}+\|\omega\|_{L^q})\le C(\|\rho\dot u\|_{L^q}+\|\omega\|_{L^q}).
\end{align}

One deduces from \eqref{f1}, \eqref{f5}, and \eqref{ff14} that, for $p\in [2, 6]$,
\begin{align}\label{z2.19}
\|F\|_{L^p}&\le C\|F\|_{L^2}^\frac{6-p}{2p}\|\nabla F\|_{L^2}^\frac{3p-6}{2p}+C\|F\|_{L^2}\nonumber\\
&\le C\|\rho\dot{u}\|_{L^2}^\frac{3p-6}{2p}
(\|\nabla u\|_{L^2}+\|P-\bar{P}\|_{L^2})^\frac{6-p}{2p}+C(\|\nabla u\|_{L^2}+\|P-\bar{P}\|_{L^2}).
\end{align}
By \eqref{f1} and \eqref{ff16}, we get that
\begin{align}\label{z2.20}
\|\omega\|_{L^p}&\le C\|\omega\|_{L^2}^\frac{6-p}{2p}\|\nabla\omega\|_{L^2}^\frac{3p-6}{2p}
+C\|\omega\|_{L^2}\nonumber\\
&\le C\big(\|\rho\dot{u}\|_{L^2}+\|\nabla u\|_{L^2}\big)^\frac{3p-6}{2p}\|\nabla u\|_{L^2}^\frac{6-p}{2p}
+C\|\nabla u\|_{L^2}\nonumber\\
&\le C\|\rho\dot{u}\|_{L^2}^\frac{3p-6}{2p}\|\nabla u\|_{L^2}^\frac{6-p}{2p}
+C\|\nabla u\|_{L^2},
\end{align}
which together with \eqref{z2.19} leads to
\begin{align}
\|F\|_{L^p}+\|\omega\|_{L^p}&\le C\big(\|\rho\dot{u}\|_{L^2}+\|P-\bar{P}\|_{L^2}+\|\nabla u\|_{L^2}\big).
\end{align}

By \eqref{ff16}, \eqref{z2.20}, Young's inequality, and H\"older's inequality, we derive that, for $p\in [2, 6]$,
\begin{align}\label{2.32}
\|\nabla\omega\|_{L^p}&\le C\big(\|\rho\dot{u}\|_{L^p}+\|\omega\|_{L^p}\big)\nonumber\\
&\le C\big(\|\rho\dot{u}\|_{L^p}+\|\rho\dot{u}\|_{L^2}
+\|\omega\|_{L^2}\big)\nonumber\\
&\le C\big(\|\rho\dot{u}\|_{L^p}+\|\nabla u\|_{L^2}\big),
\end{align}
which implies \eqref{f8}.
Moreover, we infer from \eqref{f6}, \eqref{f5}, \eqref{z2.19}, and \eqref{z2.20} that
\begin{align*}
\|\nabla u\|_{L^p}&\le C\big(\|{\rm div}\,u\|_{L^p}+\|\omega\|_{L^p}\big)\\
&\le C\big(\|F\|_{L^p}+\|P-\bar{P}\|_{L^p}+\|\omega\|_{L^p}\big)\\
&\le C\|\rho\dot{u}\|_{L^2}^\frac{3p-6}{2p}\big(\|\nabla u\|_{L^2}+\|P-\bar{P}\|_{L^2}\big)^\frac{6-p}{2p}+C\big(\|\nabla u\|_{L^2}+\|P-\bar{P}\|_{L^p}\big),
\end{align*}
as the desired \eqref{f12}.
\end{proof}

Finally, the following estimates on the material derivative of $u$ have been obtained in \cite[Lemma 2.10]{CL23}. We sketch the proof here for convenience of the reader.
\begin{lemma}\label{xu}
Under the assumption of Lemma \ref{l23}, there exists a positive constant $C$ depending only on $\Omega$ such that
\begin{gather}
\|\dot u\|_{L^6}\le C\big(\|\nabla\dot u\|_{L^2}+\|\nabla u\|_{L^2}^2\big),\label{d11}\\
\|\nabla\dot u\|_{L^2}\le C\big(\|\divv\dot u\|_{L^2}+\|\curl\dot u\|_{L^2}+\|\nabla u\|_{L^4}^2\big).\label{d12}
\end{gather}
\end{lemma}
\begin{proof}
Setting $u^\p\triangleq-u\times n$. Since $u\cdot n=0$ on $\partial\Omega$ and
\begin{align*}
a\cdot(b\times c)=(a\times b)\cdot c,\quad (a\times b)\times c=b(a\cdot c)-a(b\cdot c),
\end{align*}
one gets that
\begin{align*}
u^\p\times n=-(u\times n)\times n=u(n\cdot n)-n(u\cdot n)=u,
\end{align*}
which implies that
\begin{align*}
\dot u\cdot n=u\cdot\na u\cdot n=-u\cdot\na n\cdot u=-u\cdot\na n\cdot (u^\p\times n)=-(u\cdot\na n)\times u^\p\cdot n
\quad  ~\text{on}~\partial\Omega.
\end{align*}
Hence, we deduce that
\begin{align}\label{g24}
(\dot u+(u\cdot\na n)\times u^\p)\cdot n=0\quad  ~\text{on}~\partial\Omega,
\end{align}
which along with Poincar\'e's inequality yields that
\begin{align*}
\|\dot u+(u\cdot\na n)\times u^\p\|_{L^{\frac{3}{2}}}\le C\|\nabla(\dot u+(u\cdot\na n)\times u^\p)\|_{L^{\frac{3}{2}}},
\end{align*}
that is
\begin{align}\label{b5}
\|\dot{u}\|_{L^{\frac{3}{2}}}\le C\Big(\|\nabla\dot{u}\|_{L^{\frac{3}{2}}}+\|\nabla u\|_{L^{2}}^{2}\Big).
\end{align}
Then, by Sobolev's inequality, one has that
\begin{align*}
\|\dot{u}\|_{L^3}&\le C\Big(\|\nabla\dot{u}\|_{L^{\frac{3}{2}}}+\|\dot u\|_{L^{\frac{3}{2}}}\Big)\le C\big(\|\nabla\dot{u}\|_{L^{2}}+\|\nabla u\|_{L^2}^2\big),\nonumber\\
\|\dot{u}\|_{L^{6}}&\le C\big(\|\nabla\dot{u}\|_{L^{2}}+\|\dot u\|_{L^{3}}\big)\le C\big(\|\nabla\dot{u}\|_{L^{2}}+\|\nabla u\|_{L^2}^2\big),
\end{align*}
which gives \eqref{d11}.

Moreover, it follows from \eqref{g24}, \eqref{ff5}, and Lemma \ref{l23} that
\begin{align*}
\|\na\dot u\|_{L^2}&\le C\big(\|\div\dot u\|_{L^2}+\|\curl\dot u\|_{L^2}+\|\na((u\cdot\na n)\times u^\p)\|_{L^2}\big) \\
&\le C\big(\|\div\dot u\|_{L^2}+\|\curl\dot u\|_{L^2}+\|\na u\|_{L^4}^2\big),
\end{align*}
as the desired \eqref{d12}.
\end{proof}


\section{The proofs of Theorems \ref{thm1} and \ref{thm2}}\label{sec3}

This section is devoted to proving Theorems \ref{thm1} and \ref{thm2}.

\begin{lemma}\label{l31}
Under the assumptions of Theorem \ref{thm1}, there exist two positive constants $C_3$ and $\eta_3$, which are dependent on $K$,
but independent of $t$, such that for any $t\ge 0$,
\begin{align}\label{w1}
\|\rho(\cdot, t)-\bar\rho_0\|_{L^2}+\|\sqrt\rho u(\cdot, t)\|_{L^2}\le C_3e^{-\eta_3t}.
\end{align}
\end{lemma}
\begin{proof}[Proof]
Multiplying $\eqref{a1}_1$ and $\eqref{a1}_2$ by $G'(\rho)$ and $u$, respectively, and integration by parts, one gets that
\begin{align}\label{w8}
&\frac{d}{dt}\int\Big(\frac12\rho|u|^2+G(\rho)\Big)dx+(2\mu+\lambda)\int(\divv u)^2dx+\mu\int|\curl u|^2dx\nonumber\\
&=-\int\div (\rho u)G'(\rho)dx-\int u\cdot\nabla(P-\bar P)dx\nonumber\\
&=\int \rho u\cdot\nabla G'(\rho)dx-\int u\cdot\nabla Pdx=0,
\end{align}
due to
\begin{gather}
\bar\rho=\f{1}{|\Omega|}\int\rho(x, t)dx\equiv\f{1}{|\Omega|}\int\rho_0dx=\bar\rho_0,\label{w2}\\
-\Delta u=-\nabla\divv u+\curl\curl u.\label{z2}
\end{gather}

By \eqref{a7} and \eqref{w2}, there exists a positive constant $\tilde C<1$ depending only on $\gamma$, $\bar\rho_0$, and $\hat\rho$ such that, for any $\rho\in [0, 2\hat\rho]$,
\begin{gather}
\tilde C^2(\rho-\bar\rho_0)^2\le \tilde CG(\rho)\le (P-P(\bar\rho_0))(\rho-\bar\rho_0),\label{w9}\\
\|P-\bar P\|_{L^2}^2\le C\int G(\rho)dx.\label{w10}
\end{gather}
Multiplying $\eqref{a1}_2$ by $\cb[\rho-\bar\rho_0]$, we deduce from \eqref{a8} and Lemma \ref{l21} that
\begin{align*}
\int(P-P(\bar\rho_0))(\rho-\bar\rho_0)dx
&=\frac{d}{dt}\int\rho u\cdot\cb[\rho-\bar\rho_0]dx-\int\rho u\cdot\nabla\cb[\rho-\bar\rho_0]\cdot udx
-\int\rho u\cdot\cb[\rho_t]dx\nonumber\\
&\quad+\mu\int\nabla u\cdot\nabla\cb[\rho-\bar\rho_0]dx+(\mu+\lambda)\int(\rho-\bar\rho_0)\divv udx\nonumber\\
&=\frac{d}{dt}\int\rho u\cdot\cb[\rho-\bar\rho_0]dx
+C\|\sqrt\rho u\|_{L^4}^2\|\rho-\bar\rho_0\|_{L^2}+C\|\rho u\|_{L^2}^2\nonumber\\
&\quad+C\|\rho-\bar\rho_0\|_{L^2}\|\nabla u\|_{L^2}\nonumber\\
&\le \frac{d}{dt}\int\rho u\cdot\cb[\rho-\bar\rho_0]dx
+C\|\sqrt\rho u\|_{L^2}^\frac12\|\sqrt\rho u\|_{L^6}^\frac32\|\rho-\bar\rho_0\|_{L^2}\nonumber\\
&\quad+C\|\rho-\bar\rho_0\|_{L^2}\|\nabla u\|_{L^2}+C\|\nabla u\|_{L^2}^2
\nonumber\\
&\le \frac{d}{dt}\int\rho u\cdot\cb[\rho-\bar\rho_0]dx+\frac{\tilde C^2}{2}\|\rho-\bar\rho_0\|_{L^2}^2+C\|\nabla u\|_{L^2}^2,
\end{align*}
which along with \eqref{w9} and \eqref{w10} leads to
\begin{align*}
\tilde C^2\|\rho-\bar\rho_0\|_{L^2}^2&\le \tilde C\int G(\rho)dx\le \int(P-P(\bar\rho_0))(\rho-\bar\rho_0)dx\nonumber\\
&\le \frac{d}{dt}\int\rho u\cdot\cb[\rho-\bar\rho_0]dx+\frac{\tilde C^2}{2}\|\rho-\bar\rho_0\|_{L^2}^2+C\|\nabla u\|_{L^2}^2.
\end{align*}
Thus, according to \eqref{f6}, we arrive at
\begin{align}\label{w13}
-\frac{d}{dt}\int\rho u\cdot\cb[\rho-\bar\rho_0]dx+\frac{\tilde C^2}{2}\|\rho-\bar\rho_0\|_{L^2}^2
\le C\big(\|\divv u\|_{L^2}^2+\|\curl u\|_{L^2}^2\big).
\end{align}

Now we choose a positive constant $D_1$ suitably large and define the temporal energy functional
\begin{align*}
\cm_1(t)=D_1\int\Big(\frac12\rho|u|^2+G(\rho)\Big)dx-\int\rho u\cdot\cb[\rho-\bar\rho_0]dx.
\end{align*}
It follows from \eqref{w9} that
\begin{align*}
\left|\int\rho u\cdot\cb[\rho-\bar\rho_0]dx\right|\le \tilde C_2\Big(\frac12\|\sqrt\rho u\|_{L^2}^2+\int G(\rho)dx\Big).
\end{align*}
Hence, $\cm_1(t)$ is equivalent to $\|(\rho-\bar\rho_0, \sqrt{\rho}u)\|_{L^2}^2$ provided we choose $D_1$ large enough. Taking a linear combination of \eqref{w8} and \eqref{w13}, we deduce that, for any $t\ge 0$,
\begin{align}\label{w15}
\frac{d}{dt}\cm_1(t)+\frac{\cm_1(t)}{D_1}+\frac{\|\divv u\|_{L^2}^2+\|\curl u\|_{L^2}^2}{D_1}\le 0.
\end{align}
Integrating \eqref{w15} with respect to the time over $[0, t]$ gives \eqref{w1}.
\end{proof}

Next, we establish the time-decay rate of $\|\nabla u\|_{L^2}$.
\begin{lemma}
Under the assumptions of Theorem \ref{thm1}, there exist two positive constant $C_4$ and $\eta_4$, which are dependent on $K$, but independent of $t$, such that for any $t\ge 0$,
\begin{align}\label{w17}
\|\nabla u(\cdot, t)\|_{L^2}\le C_4e^{-\eta_4t}.
\end{align}
\end{lemma}
\begin{proof}[Proof]
By the definition of the material derivative and \eqref{z2}, we can rewrite $\eqref{a1}_2$ as
\begin{align}\label{w18}
\rho\dot u+\nabla(P-\bar P)=(2\mu+\lambda)\nabla\divv u-\mu\curl\curl u.
\end{align}
Multiplying \eqref{w18} by $\dot u$ and integration by parts, one gets that
\begin{align}\label{w19}
\int\rho|\dot u|^2dx=-\int\dot u\cdot\nabla(P-\bar P)dx+(2\mu+\lambda)\int\nabla\divv u\cdot\dot udx
-\mu\int\curl\curl u\cdot\dot udx
\triangleq\sum_{i=1}^3I_i.
\end{align}

By $\eqref{a1}_1$ and $P=\rho^\gamma$, we have
\begin{align*}
P_t+\div(Pu)+(\gamma-1)P\divv u=0,
\end{align*}
which gives that
\begin{equation}\label{z3.26}
(P-\Bar{P})_t+u\cdot\nabla(P-\Bar{P})+\gamma P\divv u-(\gamma-1)\overline{P\divv u}=0.
\end{equation}
Owing to
\begin{align*}
\overline{P\divv u}=\frac{1}{|\Omega|}\int \rho^\gamma\divv udx\leq C(\hat{\rho},\gamma,\Omega)\left\|\divv u\right\|_{L^2},
\end{align*}
then it follows from \eqref{z3.26} that
\begin{align}\label{z4}
I_1&=-\int u_t\cdot\nabla(P-\bar P)dx-\int u\cdot\nabla u\cdot\nabla(P-\bar P)dx\nonumber\\
&=\frac{d}{dt}\int(P-\bar P)\divv udx-\int\divv u(P-\bar P)_tdx-\int u\cdot\nabla u\cdot\nabla (P-\bar P)dx\nonumber\\
&=\frac{d}{dt}\int(P-\bar P)\divv udx-\gamma\int P(\divv u)^2dx-\int u\cdot\na(P-\bar P)\divv udx\nonumber\\
&\quad+(\gamma-1)\int\divv u\overline{P\divv u}dx-\int u\cdot\nabla u\cdot\nabla Pdx\nonumber\\
&=\frac{d}{dt}\int(P-\bar P)\divv udx-\gamma\int P(\divv u)^2dx+\int (P-\bar P)\nabla u:\nabla udx\nonumber\\
&\quad+(\gamma-1)\int\divv u\overline{P\divv u}dx-\int_{\partial\Omega}Pu\cdot\nabla u\cdot ndS\nonumber\\
&\le \frac{d}{dt}\int(P-\bar P)\divv udx+C\|\nabla u\|_{L^2}^2,
\end{align}
where we have used
\begin{align*}
&\int u\cdot\na(P-\bar P)\divv udx-\int u\cdot\nabla u\cdot\nabla (P-\bar P)dx \notag \\
&=\int u^i\pa_i(P-\bar P)\pa_ju^jdx-\int u^i\partial_iu^j\partial_j(P-\bar P)dx\nonumber\\
&=-\int\pa_iu^i\pa_ju^j(P-\bar P)dx-\int(P-\bar P)\pa_j\pa_iu^ju^idx-\int u^i\partial_iu^j\partial_j(P-\bar P)dx\nonumber\\
&=-\int (P-\bar P)(\divv u)^2dx+\int (P-\bar P)\partial_iu^j\partial_ju^idx-\int_{\partial\Omega}(P-\bar P)u^i\partial_iu^jn^jdS\nonumber\\
&=-\int (P-\bar P)(\divv u)^2dx+\int (P-\bar P)\nabla u:\nabla udx+\int_{\partial\Omega}(P-\bar P)u\cdot\nabla n\cdot udS  \\
&\le C\int |P-\bar P||\nabla u|^2dx+C\int_{\partial\Omega}|u|^2dS
\le C\|\nabla u\|_{L^2}^2,
\end{align*}
due to \eqref{a4}, \eqref{a15}, the trace theorem, \eqref{z2.5}, and \eqref{a9}.

By virtue of \eqref{a4} and \eqref{a15}, we infer from integration by parts that
\begin{align}\label{z5}
I_2&=(2\mu+\lambda)\int_{\partial\Omega}{\rm div}\,u(\dot{u}\cdot n)dS-(2\mu+\lambda)\int{\rm div}\,u{\rm div}\,\dot{u}dx\nonumber\\
&=(2\mu+\lambda)\int_{\partial\Omega}{\rm div}\,u(u\cdot\nabla u\cdot n)dS-\frac{2\mu+\lambda}{2}\frac{d}{dt}\int({\rm div}\,u)^2dx
-(2\mu+\lambda)\int{\rm div}\,u{\rm div}\,(u\cdot\nabla u)dx\nonumber\\
&=-\frac{2\mu+\lambda}{2}\frac{d}{dt}\int(\divv u)^2dx
-(2\mu+\lambda)\int_{\partial\Omega}\divv u(u\cdot\nabla n\cdot u)dS
-(2\mu+\lambda)\int{\rm div}\,u\partial_j(u^i\partial_iu^j)dx\nonumber\\
&=-\frac{2\mu+\lambda}{2}\frac{d}{dt}\int({\rm div}\,u)^2dx
-(2\mu+\lambda)\int_{\partial\Omega}{\rm div}\,u(u\cdot\nabla n\cdot u)dS\nonumber\\
&\quad-(2\mu+\lambda)\int{\rm div}\,u\nabla u:\nabla udx-(2\mu+\lambda)\int{\rm div}\,uu^j\partial_j\partial_{i}u^idx\nonumber\\
&=-\frac{2\mu+\lambda}{2}\frac{d}{dt}\int({\rm div}\,u)^2dx
-(2\mu+\lambda)\int_{\partial\Omega}{\rm div}\,u(u\cdot\nabla n\cdot u)dS\nonumber\\
&\quad-(2\mu+\lambda)\int{\rm div}\,u\nabla u:\nabla udx+\frac{2\mu+\lambda}{2}\int({\rm div}\,u)^3dx
\nonumber\\
&\le -\frac{2\mu+\lambda}{2}\frac{d}{dt}\int({\rm div}\,u)^2dx
+\frac14\|\sqrt{\rho}\dot{u}\|_{L^2}^2+C\big(\|\nabla u\|_{L^2}^4+\|\nabla u\|_{L^2}^2+\|\nabla u\|_{L^3}^3\big),
\end{align}
where we have used
\begin{align*}
\int{\rm div}\,uu^j\partial_j\partial_iu^idx&=-\int\partial_j(\partial_ku^ku^j)\partial_iu^idx=-\int\partial_j\partial_ku^ku^j\partial_iu^idx
-\int\div  u\partial_ju^j\partial_iu^idx,
\end{align*}
and
\begin{align}\label{3.24}
&\Big|-(2\mu+\lambda)\int_{\partial\Omega}{\rm div}\,u(u\cdot\nabla n\cdot u)dS\Big|\nonumber\\
&=\Big|-\int_{\partial\Omega}\Big(F+(P-\bar P)\Big)(u\cdot\nabla n\cdot u)dS\Big|\nonumber\\
&\le \Big|\int_{\partial\Omega}F(u\cdot\nabla n\cdot u)dS\Big|+\Big|\int_{\partial\Omega}(P-\bar P)(u\cdot\nabla n\cdot u)dS\Big|
\nonumber\\
&\le C\int_{\partial\Omega}|F||u|^2dS+C\int_{\partial\Omega}|u|^2dS
\nonumber\\
&\le C\big(\|\nabla F\|_{L^2}\|u\|_{L^4}^2+\|F\|_{L^6}\|u\|_{L^3}\|\nabla u\|_{L^2}
+\|F\|_{L^2}\|u\|_{L^4}^2\big)+C\|\nabla u\|_{L^2}^2\nonumber\\
&\le C\|F\|_{H^1}\|u\|_{H^1}^2+C\|\nabla u\|_{L^2}^2+C\|\nabla u\|_{L^2}^4\nonumber\\
&\le \frac14\|\sqrt{\rho}\dot{u}\|_{L^2}^2
+C\big(\|\nabla u\|_{L^2}^4+\|\nabla u\|_{L^2}^2\big),
\end{align}
due to the trace theorem, \eqref{z2.5}, \eqref{f6}, and \eqref{f8}.

Noting that
\begin{align*}
\int{\rm curl}\,u\cdot(u^i\partial_i{\rm curl}\,u)dx=-\int{\rm curl}\,u\cdot(u^i\partial_i{\rm curl}\,u)dx-\int|{\rm curl}\,u|^2{\rm div}\,udx,
\end{align*}
we have
\begin{align*}
\int{\rm curl}\,u\cdot(u^i\partial_i{\rm curl}\,u)dx
=-\frac12\int|{\rm curl}\,u|^2{\rm div}\,udx.
\end{align*}
This implies that
\begin{align*}
\int{\rm curl}\,u\cdot{\rm curl}\,(u\cdot\nabla u)dx&=\int{\rm curl}\,u\cdot{\rm curl}\,(u^i\partial_i u)dx
=\int{\rm curl}\,u\cdot\big(u^i{\rm curl}\,\partial_iu+\nabla u^i\times\partial_iu\big)dx\nonumber\\
&=-\int\partial_i({\rm curl}\,u u^i){\rm curl}\,udx+\int(\nabla u^i\times\partial_iu)\cdot{\rm curl}\,udx\nonumber\\
&=-\frac12\int|{\rm curl}\,u|^2{\rm div}\,udx+\int(\nabla u^i\times\partial_iu)\cdot{\rm curl}\,udx,
\end{align*}
which combined with \eqref{a4} and integration by parts leads to
\begin{align}\label{3.26}
I_3&=-\mu\int{\rm curl}\,u\cdot{\rm curl}\,\dot{u}dx\nonumber\\
&=-\frac{\mu}{2}\frac{d}{dt}\int|{\rm curl}\,u|^2dx
-\mu\int{\rm curl}\,u\cdot{\rm curl}\,(u\cdot\nabla u)dx\nonumber\\
&=-\frac{\mu}{2}\frac{d}{dt}\int|{\rm curl}\,u|^2dx
-\mu\int(\nabla u^i\times\partial_iu)\cdot{\rm curl}\,udx
+\frac{\mu}{2}\int|{\rm curl}\,u|^2{\rm div}\,udx\nonumber\\
&\le -\frac{\mu}{2}\frac{d}{dt}\int|{\rm curl}\,u|^2dx
+C\|\nabla u\|_{L^3}^3.
\end{align}
Substituting \eqref{z4}, \eqref{z5}, and \eqref{3.26} into \eqref{w19} yields that
\begin{align}\label{w30}
\frac{d}{dt}\Psi(t)+\frac34\|\sqrt{\rho}\dot{u}\|_{L^2}^2
\le C\big(\|\nabla u\|_{L^3}^3+\|\nabla u\|_{L^2}^4+\|\nabla u\|_{L^2}^2\big),
\end{align}
where
\begin{align*}
\Psi(t)&\triangleq\int\Big[\frac12\big((2\mu+\lambda)(\divv u)^2+\mu|\curl u|^2\big)-(P-\bar P)\divv u\Big]dx.
\end{align*}

By \eqref{f12}, one has that
\begin{align}\label{w31}
\|\nabla u\|_{L^3}^3&\le C\|\sqrt{\rho}\dot{u}\|_{L^2}^\frac{3}{2}\big(\|\nabla u\|_{L^2}+\|P-\bar P\|_{L^2}\big)^\frac32
+C\big(\|\nabla u\|_{L^2}^3+\|P-\bar P\|_{L^3}^3\big)\nonumber\\
&\le  C\|\sqrt{\rho}\dot{u}\|_{L^2}^\frac{3}{2}\big(\|\nabla u\|_{L^2}+\|\rho-\bar\rho_0\|_{L^2}\big)^\frac32
+C\big(\|\nabla u\|_{L^2}^3
+\|\rho-\bar\rho_0\|_{L^2}^2\big)\nonumber\\
&\le \f14\|\sqrt{\rho}\dot{u}\|_{L^2}^2+C\big(\|\nabla u\|_{L^2}^2+\|\nabla u\|_{L^2}^6\big)
+C\|\rho-\bar\rho_0\|_{L^2}^2,
\end{align}
due to
\begin{align*}
\|P-\bar P\|_{L^3}\le C\|P-\bar P\|_{L^\infty}^\frac13\bigg(\int|P-\bar P|^2dx\bigg)^\frac13\le C(\hat{\rho})\|\rho-\bar\rho_0\|_{L^2}^\frac23.
\end{align*}
Plugging \eqref{w31} into \eqref{w30}, we arrive at
\begin{align}\label{lwz3}
\frac{d}{dt}\Psi(t)+\frac12\|\sqrt{\rho}\dot{u}\|_{L^2}^2
\le C\big(\|\nabla u\|_{L^2}^2+\|\nabla u\|_{L^2}^6+\|\rho-\bar\rho_0\|_{L^2}^2\big)
\end{align}
due to $\|\nabla u\|_{L^2}^4\leq \|\nabla u\|_{L^2}^2+\|\nabla u\|_{L^2}^6$.
In view of \eqref{a8}, \eqref{a9}, \eqref{w1}, and \eqref{w8}, it is clear that
\begin{align}\label{w33}
\sqrt{\rho}\dot u\in L_{\loc}^2(0,\infty; L^2), \quad
 \Psi(t)\in C[0, \infty),
\end{align}
which along with \eqref{w1} and \eqref{w8} implies that
\begin{align}\label{zzz}
&\int_0^\infty\int\Big[\frac12\big((2\mu+\lambda)(\divv u)^2+\mu|\curl u|^2\big)
-(P-\bar P)\divv u+D_2|\rho-\bar\rho_0|^2\Big]dxdt<\infty.
\end{align}

Next, for $t\ge 0$, we choose a positive constant $D_3$ suitably large and define the temporal energy functional
\begin{align}\label{w35}
\cm_2(t)&=D_3\cm_1(t)+\int\Big[\frac12\big((2\mu+\lambda)(\divv u)^2+\mu|\curl u|^2\big)-(P-\bar P)\divv u+D_2|\rho-\bar\rho_0|^2\Big](t)dx.
\end{align}
Noticing that $\cm_2(t)$ is equivalent to $\|(\rho-\bar\rho_0, \sqrt\rho u, \nabla u)(t)\|_{L^2}^2$ provided we choose $D_2$ and $D_3$ large enough. Fix a positive constant $\delta_1$ that may be small. In light of \eqref{w1}
and \eqref{zzz}, we conclude that there exists a positive constant $T_1>0$ such that
\begin{align}
\cm_2(T_1)<\delta_1.
\end{align}

Now we claim that, for any $t\ge T_1$,
\begin{align}\label{b35}
\Psi(t)+D_2\|(\rho-\bar\rho_0)(\cdot,t)\|_{L^2}^2<2\delta_1.
\end{align}
This implies that, for any $t\ge T_1$,
\begin{align}\label{b36}
\int\big[\mu|\nabla u|^2+(\mu+\lambda)(\divv u)^2\big](t)dx<4\delta_1.
\end{align}
Let $\delta_1$ be small enough, then one deduces from \eqref{w15} and \eqref{b35} that, for any $t\ge T_1$,
\begin{align}\label{b37}
\frac{d}{dt}\cm_2(t)+\frac{\cm_2(t)}{D_3}+\frac{\|\sqrt\rho\dot{u}\|_{L^2}^2}{D_3}\le 0.
\end{align}
Integrating \eqref{b37} with respect to $t$ over $[T_1, \infty)$ implies \eqref{w17}.

It remains to prove \eqref{b35}. If \eqref{b35} were false, taking advantage of \eqref{w33}, there exists a time $T_2>T_1$ such that
\begin{align}\label{b38}
\Psi(T_2)+D_2\|(\rho-\bar\rho_0)(\cdot,T_2)\|_{L^2}^2=2\delta_1.
\end{align}
Taking a minimal value of $T_2$ satisfying \eqref{b38}, then \eqref{b35} holds for any $T_1\le t<T_2$. Integrating \eqref{b37} from
$T_1$ to $T_2$ implies that
\begin{align*}
\cm_2(T_2)\le \cm_2(T_1)<\delta_1,
\end{align*}
which contradicts to \eqref{b38}. Hence, \eqref{b35} holds for any $t\ge T_1$.
\end{proof}

\begin{lemma}
Under the assumptions of Theorem \ref{thm1}, there exist two positive constant $C_5$ and $\eta_5$, which are dependent on
 $K$, but independent of $t$, such that for any $t\ge 0$,
\begin{align}\label{w41}
\sup_{t\ge 0}\|\sqrt\rho\dot u\|_{L^2}^2
+\int_0^\infty\|\nabla\dot u\|_{L^2}^2dt\le C,
\end{align}
and
\begin{align}\label{w42}
\sup_{t\ge 0}\|\sqrt\rho\dot u\|_{L^2}^2\le C_5e^{-\eta_5t}.
\end{align}
\end{lemma}
\begin{proof}[Proof]
Taking \eqref{f5}, we rewrite $\eqref{a1}_2$ as
\begin{align}\label{3.41}
\rho\dot{u}=\nabla F-\mu\curl{\rm curl}\,u.
\end{align}
Applying $\dot{u}^j[\partial/\partial t+\div (u\cdot)]$ to the $j$th-component of $\eqref{3.41}$,
 and then integrating the resulting equality over $\Omega$, one gets that
\begin{align}\label{w44}
\frac12\frac{d}{dt}\int\rho|\dot{u}|^2dx
&=\int\big(\dot{u}\cdot\nabla F_t+\dot{u}^j{\rm div}\,(u\partial_jF)\big)dx
-\mu\int\big(\dot{u}\cdot\curl{\rm curl}\,u_t+\dot{u}^j{\rm div}\,((\curl{\rm curl}\,u)^ju)\big)dx\nonumber\\
&\triangleq J_1+J_2.
\end{align}
We denote by $h\triangleq u\cdot(\nabla n+(\nabla n)^{tr})$ and $u^\bot\triangleq-u\times n$, then it deduces from Lemma \ref{l23} that
\begin{align}\label{w43}
&-\int_{\partial\Omega}F_t(u\cdot\nabla n\cdot u)dS\nonumber\\
&=-\frac{d}{dt}\int_{\partial\Omega}(u\cdot\nabla n\cdot u)FdS
+\int_{\partial\Omega}Fh\cdot\dot{u}dS-\int_{\partial\Omega}Fh\cdot(u\cdot\nabla u)dS\nonumber\\
&=-\frac{d}{dt}\int_{\partial\Omega}(u\cdot\nabla n\cdot u)FdS
+\int_{\partial\Omega}Fh\cdot\dot{u}dS-\int_{\partial\Omega}Fh^i(\nabla u^i\times u^\bot)\cdot ndS\nonumber\\
&=-\frac{d}{dt}\int_{\partial\Omega}(u\cdot\nabla n\cdot u)FdS
+\int_{\partial\Omega}Fh\cdot\dot{u}dS
+\int Fh^i\nabla\times u^\bot\cdot\nabla u^idx
-\int\nabla u^i\times u^\bot\cdot\nabla(Fh^i)dx\nonumber\\
&\le-\frac{d}{dt}\int_{\partial\Omega}(u\cdot\nabla n\cdot u)FdS
+C\|\nabla F\|_{L^2}\|u\|_{L^3}\|\dot{u}\|_{L^6}\nonumber\\
&\quad+C\big(\|F\|_{L^3}\|u\|_{L^6}\|\nabla\dot{u}\|_{L^2}
+\|F\|_{L^3}\|u\|_{L^6}\|\dot{u}\|_{L^2}
+\|F\|_{L^3}\|\nabla u\|_{L^2}\|\dot{u}\|_{L^6}\big)\nonumber\\
&\quad+C\big(\|\nabla u\|_{L^2}\|u\|_{L^6}^2\|\nabla F\|_{L^6}+\|\nabla u\|_{L^4}^2\|u\|_{L^6}\|F\|_{L^3}\big)\nonumber\\
&\le -\frac{d}{dt}\int_{\partial\Omega}(u\cdot\nabla n\cdot u)FdS
+C\big(\|\rho\dot{u}\|_{L^2}+\|P-\bar P\|_{L^2}+\|\nabla u\|_{L^2}+\|\nabla u\|_{L^2}^2\big)\nonumber\\
&\quad\cdot\|\nabla u\|_{L^2}\big(\|\nabla\dot{u}\|_{L^2}+\|\nabla u\|_{L^2}^2+\|\nabla u\|_{L^4}^2\big)
+C\|\nabla u\|_{L^2}^3\|\nabla F\|_{L^6}\nonumber\\
&\le -\frac{d}{dt}\int_{\partial\Omega}(u\cdot\nabla n\cdot u)FdS
+C\|\nabla u\|_{L^2}^3\|\nabla F\|_{L^6}+\delta\|\nabla\dot{u}\|_{L^2}^2\nonumber\\
&\quad+C\big(\|\sqrt{\rho}\dot{u}\|_{L^2}^2\|\nabla u\|_{L^2}^2+\|\nabla u\|_{L^2}^6+\|\nabla u\|_{L^4}^4
+\|\nabla u\|_{L^2}^2\big),
\end{align}
due to
\begin{gather*}
{\rm div}(\nabla u^i\times u^\bot)=u^\bot\cdot\curl\nabla u^i-\nabla u^i\cdot\curl u^\bot=-\nabla u^i\cdot\curl u^\bot,\\
\|\dot{u}\|_{L^6}\le C\big(\|\nabla\dot{u}\|_{L^2}+\|\nabla u\|_{L^2}^2\big),\\
\|\sqrt\rho\dot u\|_{L^2}\le C\|\dot u\|_{L^2}\le C\|\nabla\dot u\|_{L^2},
\\
\|\nabla F\|_{L^6}\le C\|\rho\dot u\|_{L^6}\le C\|\dot u\|_{L^6}
\le C(\|\nabla\dot u\|_{L^2}+\|\nabla u\|_{L^2}^2).
\end{gather*}

Thus, it follows from integration by parts, H\"older's inequality, Gagliardo--Nirenberg inequality, \eqref{a15}, \eqref{f8}, and \eqref{z2.5} that
\begin{align}\label{lwz1}
J_1&=\int\big(\dot{u}\cdot\nabla F_t+\dot{u}^j\pa_i(\pa_jFu^i)\big)dx\nonumber\\
&=\int_{\pa\Omega}F_t\dot u\cdot ndS-\int F_t\divv\dot udx-\int u^i\pa_i\dot u^j\partial_jFdx\nonumber\\
&=-\int_{\pa\Omega}F_t(u\cdot\nabla n\cdot u)dS-(2\mu+\lambda)\int(\divv\dot u)^2dx
+(2\mu+\lambda)\int\divv\dot u\nabla u:\nabla udx\nonumber\\
&\quad-\gamma\int P\divv\dot u\divv udx+\int\divv\dot uu\cdot\nabla Fdx-\int u^i\pa_i\dot u^j\partial_jFdx
\nonumber\\
&\le -\int_{\pa\Omega}F_t(u\cdot\nabla n\cdot u)dS-(2\mu+\lambda)\|\divv\dot u\|_{L^2}^2+\delta\|\nabla\dot u\|_{L^2}^2
+C\|\nabla u\|_{L^4}^4\nonumber\\
&\quad+C\|u\|_{L^6}^2\|\nabla F\|_{L^3}^2+C\|\nabla u\|_{L^2}^2\nonumber\\
&\le -\int_{\pa\Omega}F_t(u\cdot\nabla n\cdot u)dS-(2\mu+\lambda)\|\divv\dot u\|_{L^2}^2+\delta\|\nabla\dot u\|_{L^2}^2
+C\|\nabla u\|_{L^4}^4\nonumber\\
&\quad+C\|\nabla u\|_{L^2}^2\|\nabla F\|_{L^2}\|\nabla F\|_{L^6}
+C\|\nabla u\|_{L^2}^2\nonumber\\
&\le -\int_{\pa\Omega}F_t(u\cdot\nabla n\cdot u)dS-(2\mu+\lambda)\|\divv\dot u\|_{L^2}^2+\delta\|\nabla\dot u\|_{L^2}^2
\nonumber\\
&\quad+C\|\nabla u\|_{L^4}^4
+C\|\nabla u\|_{L^2}^2\|\nabla F\|_{L^2}\|\nabla F\|_{L^6}
+C\|\nabla u\|_{L^2}^2\nonumber\\
&\le -\frac{d}{dt}\int_{\pa\Om}(u\cdot\na n\cdot u)FdS-(2\mu+\lambda)
\|\divv\dot u\|_{L^2}^2+\delta\|\nabla\dot u\|_{L^2}^2
\nonumber\\
&\quad
+C\big(\|\sqrt{\rho}\dot{u}\|_{L^2}^2\|\nabla u\|_{L^2}^2+\|\nabla u\|_{L^2}^6+\|\nabla u\|_{L^4}^4+\|\nabla u\|_{L^2}^2\big)\nonumber\\
&\quad+C(\|\na u\|_{L^2}^3+\|\na u\|_{L^2}^2\|\na F\|_{L^2})\|\na F\|_{L^6}\nonumber\\
&\le -\frac{d}{dt}\int_{\pa\Om}(u\cdot\na n\cdot u)FdS-(2\mu+\lambda)
\|\divv\dot u\|_{L^2}^2+\delta\|\nabla\dot u\|_{L^2}^2
\nonumber\\
&\quad
+C\big(\|\sqrt{\rho}\dot{u}\|_{L^2}^2\|\nabla u\|_{L^2}^2+\|\nabla u\|_{L^2}^6+\|\nabla u\|_{L^4}^4+\|\nabla u\|_{L^2}^2\big)
+C\|\na u\|_{L^2}^4\|\sqrt\r\dot u\|_{L^2}^2,
\end{align}
where we have used
\begin{align*}
\|\na u\|_{L^2}^3\|\na F\|_{L^6}
&\le \|\na u\|_{L^2}^3(\|\na \dot u\|_{L^2}+\|\na u\|_{L^2}^2)\nonumber\\
&\le \delta\|\na\dot u\|_{L^2}^2+C\|\na u\|_{L^2}^6+C\|\na u\|_{L^2}^4\notag \\
\|\na u\|_{L^2}^2\|\na F\|_{L^2}\|\na F\|_{L^6}
&\le C\|\na u\|_{L^2}^2\|\sqrt\r\dot u\|_{L^2}(\|\na\dot u\|_{L^2}+\|\na u\|_{L^2}^2)\nonumber\\
&\le \delta\|\na\dot u\|_{L^2}^2+C\|\na u\|_{L^2}^4\|\sqrt\r\dot u\|_{L^2}^2
+C\|\na u\|_{L^2}^6+C\|\na u\|_{L^2}^4,
\end{align*}
and
\begin{align*}
F_{t}+u\cdot\nabla F&=(2\mu+\lambda)\divv u_t-P_t+(2\mu+\lambda)u\cdot\nabla\divv u-u\cdot\nabla P\nonumber\\
&=(2\mu+\lambda)\divv\dot{u}-(2\mu+\lambda)\divv(u\cdot\nabla u)+(2\mu+\lambda)u\cdot\nabla\divv u-(P_t+u\cdot\nabla P)\nonumber\\
&=(2\mu+\lambda)\divv\dot{u}-(2\mu+\lambda)\partial_i(u^j\partial_ju^i)
+(2\mu+\lambda)u^j\partial_j\partial_iu^i
+\gamma P\divv u\nonumber\\
&=(2\mu+\lambda)\divv \dot{u}-(2\mu+\lambda)\partial_iu^j\partial_ju^i
+\gamma P\divv u.
\end{align*}

By direct calculations, we have
\begin{align}\label{lwz2}
J_2&=-\mu\int \dot{u}\cdot(\curl{\rm curl}\,u_t)dx-\mu\int\dot{u}\cdot(\curl{\rm curl}\,u){\rm div}\,udx
-\mu\int u^i\dot{u}\cdot\curl(\partial_i{\rm curl}\,u)dx\nonumber\\
&=-\mu\int |{\rm curl}\,\dot{u}|^2dx+\mu\int {\rm curl}\,\dot{u}\cdot{\rm curl}\,(u\cdot\nabla u)dx+\mu\int ({\rm curl}\,u\times\dot{u})\cdot\nabla{\rm div}\,udx\nonumber\\
&\quad
-\mu\int {\rm div}\,u({\rm curl}\,u\cdot{\rm curl}\,\dot{u})dx-\mu\int u^i{\rm div}\,(\partial_i{\rm curl}\,u\times\dot{u})dx
-\mu\int u^i\partial_i{\rm curl}\,u\cdot{\rm curl}\,\dot{u}dx\nonumber\\
&=-\mu\int |{\rm curl}\,\dot{u}|^2dx+\mu\int {\rm curl}\,\dot{u}\partial_iu\times\nabla u^idx+\mu\int ({\rm curl}\,u\times\dot{u})\cdot\nabla{\rm div}\,udx\nonumber\\
&\quad
-\mu\int {\rm div}\,u({\rm curl}\,u\cdot{\rm curl}\,\dot{u})dx-\mu\int u^i{\rm div}\,(\partial_i{\rm curl}\,u\times\dot{u})dx\nonumber\\
&=-\mu\int |{\rm curl}\,\dot{u}|^2dx+\mu\int {\rm curl}\,\dot{u}\partial_iu\times\nabla u^idx+\mu\int ({\rm curl}\,u\times\dot{u})\cdot\nabla{\rm div}\,udx
\nonumber\\
&\quad-\mu\int {\rm div}\,u({\rm curl}\,u\cdot{\rm curl}\,\dot{u})dx-\mu\int  u\cdot\nabla{\rm div}\,({\rm curl}\,u\times\dot{u})dx
+\mu\int  u^i{\rm div}\,({\rm curl}\,u\times\partial_i\dot{u})dx\nonumber\\
&=-\mu\int |{\rm curl}\,\dot{u}|^2dx+\mu\int {\rm curl}\,\dot{u}\nabla_iu\times\nabla u^idx\nonumber\\
&\quad-\mu\int {\rm div}\,u({\rm curl}\,u\cdot{\rm curl}\,\dot{u})dx
-\mu\int \nabla u^i\cdot({\rm curl}\,u\times\partial_i\dot{u})dx\nonumber\\
&\le \delta \|\nabla\dot{u}\|_{L^2}^2+C \|\nabla u\|_{L^4}^4-\mu \|{\rm curl}\,\dot{u}\|_{L^2}^2,
\end{align}
due to
\begin{align*}
\curl(\dot{u}{\rm div}\,u)&={\rm div}\,u{\rm curl}\,\dot{u}+\nabla{\rm div}\,u\times\dot{u},\\
{\rm div}\,(\partial_i{\rm curl}\,u\times\dot{u})&=\dot{u}\cdot\curl(\partial_i{\rm curl}\,u)-\partial_i{\rm curl}\,u\cdot\curl\dot{u},\\
\int{\rm curl}\,u\cdot(\nabla{\rm div}\,u\times\dot{u})dx&=-\int({\rm curl}\,u\times\dot{u})\cdot\nabla{\rm div}\,udx,\\
\int{\rm curl}\,\dot{u}\cdot{\rm curl}\,(u\cdot\nabla u)dx&=\int{\rm curl}\,\dot{u}\cdot{\rm curl}\,(u^i\partial_iu)dx \\
& =\int{\rm curl}\,\dot{u}\big(u^i{\rm curl}\,\partial_iu+\partial_iu\cdot\nabla u^i\big)dx\nonumber\\
&=\int u^i\partial_i{\rm curl}\,u\cdot{\rm curl}\,\dot{u}dx+\int{\rm curl}\,\dot{u}\partial_iu\times\nabla u^idx,
\end{align*}
and
\begin{align*}
\int u\cdot\nabla{\rm div}\,({\rm curl}\,u\times\dot{u})dx
&=\int u^i\partial_i{\rm div}\,({\rm curl}\,u\times\dot{u})dx\nonumber\\
&=\int u^i\partial_i(\dot{u}\cdot\curl{\rm curl}\,u-{\rm curl}\,\dot{u}\cdot{\rm curl}\,u)dx\nonumber\\
&=\int u^i(\dot{u}\cdot\partial_i\curl{\rm curl}\,u-{\rm curl}\,\dot{u}\cdot\partial_i{\rm curl}\,u)dx\nonumber\\
&\quad+\int u^i(\partial_i\dot{u}\cdot\curl{\rm curl}\,u-\partial_i{\rm curl}\,\dot{u}\cdot{\rm curl}\,u)dx\nonumber\\
&=\int u^i{\rm div}\,(\partial_i{\rm curl}\,u\times\dot{u})dx+\int u^i{\rm div}\,({\rm curl}\,u\times\partial_i\dot{u})dx.
\end{align*}

Plugging \eqref{lwz1} and \eqref{lwz2} into \eqref{w44}, we obtain after choosing $\delta$ suitably small that
\begin{align}\label{w51}
&\frac{d}{dt}\bigg(\|\sqrt\r\dot u\|_{L^2}^2+\int_{\pa\Om}(u\cdot\na n\cdot u)FdS
\bigg)+\|\na\dot u\|_{L^2}^2\nonumber\\
&\le C\big(\|\na u\|_{L^2}^4\|\sqrt\r\dot u\|_{L^2}^2+\|\sqrt{\rho}\dot{u}\|_{L^2}^2\|\nabla u\|_{L^2}^2+\|\nabla u\|_{L^2}^6+\|\nabla u\|_{L^4}^4
+\|\nabla u\|_{L^2}^2\big)\nonumber\\
&\le C\big(\|\sqrt\rho\dot u\|_{L^2}^2+\|\na u\|_{L^2}^2\big)\|\sqrt\rho\dot u\|_{L^2}^2+C\big(\|\nabla u\|_{L^2}^2+\|\rho-\bar\rho_0\|_{L^2}^2\big),
\end{align}
due to
\begin{align}
\|\na u\|_{L^4}^4&\le C\|\r\dot u\|_{L^2}^3\big(\|\na u\|_{L^2}+\|P-\bar P\|_{L^2}\big)+C\big(\|\na u\|_{L^2}+\|P-\bar P\|_{L^4}\big)^4\nonumber\\
&\le C\big(\|\sqrt\r\dot u\|_{L^2}^4+\|\na u\|_{L^2}^4+\|\rho-\bar\rho_0\|_{L^2}^2\big).
\end{align}
It follows from the trace theorem, \eqref{z2.5}, and \eqref{w17} that
\begin{align}\label{w59}
\Big|\int_{\partial\Omega}(u\cdot\nabla n\cdot u)FdS\Big|
&\le C\||u|^2|F|\|_{W^{1,1}}\le C\|\nabla u\|_{L^2}^2\|F\|_{H^1}\nonumber\\
&\le \frac{1}{2}\|\sqrt{\rho}\dot{u}\|_{L^2}^2
+C\big(\|\nabla u\|_{L^2}^4+\|\nabla u\|_{L^2}^2\big)\nonumber\\
&\le \f12\|\sqrt\r\dot u\|_{L^2}^2+C\|\nabla u\|_{L^2}^2.
\end{align}

Choose a positive constant $D_4$ and define a temporal energy functional
\begin{align*}
\cm_3(t)\triangleq D_5\cm_2(t)+\|\sqrt{\rho}\dot{u}(\cdot,t)\|_{L^2}^2
+\int_{\partial\Omega}(u\cdot\nabla n\cdot u)FdS
+D_4\|(\rho-\bar\rho_0)(\cdot,t)\|_{L^2}^2.
\end{align*}
Noting that $\cm_3(t)$ is equivalent to $\|(\rho-\bar\rho_0, \nabla u,
\sqrt{\rho}\dot u)\|_{L^2}^2$ provided $D_4$ and $D_5$ are chosen large enough. In light
 of \eqref{b37}, we get that, for any $t\ge 0$,
\begin{align}\label{w58}
\frac{d}{dt}\cm_3(t)+\frac{\cm_3(t)}{D_5}+\frac{\mu\|\nabla\dot u\|_{L^2}^2}{D_5}\le 0.
\end{align}
Integrating \eqref{w58} with respect to $t$ over $[0, \infty)$ together with \eqref{w59} leads to \eqref{w42}.
\end{proof}

\begin{lemma}\label{l34}
Under the assumptions of Theorem \ref{thm1}, if additionally $\inf\limits_{x\in\Omega}\rho_0(x)\ge \rho_*>0$, then there exist two positive constants $C_6$ and $\eta_6$, which are dependent on
$K$, but independent of $t$, such that for any $t\ge 0$,
\begin{align}\label{w61}
\|(\rho-\bar\rho_0)(\cdot, t)\|_{L^\infty}\le C_6e^{-\eta_6t}.
\end{align}
\end{lemma}
\begin{proof}
Denote by $\vr\triangleq\rho-\bar\rho_0$, owing to $\int Fdx=0$,
then it follows from \eqref{f2}, Lemmas \ref{l24}, \ref{xu}, \eqref{w42}, \eqref{w1}, and \eqref{w17} that
\begin{align}\label{3.53}
\|F\|_{L^\infty}^2&\le C\|F\|_{L^{6}}\|\nabla F\|_{L^6}\nonumber\\
&\le C\big(\|\nabla u\|_{L^2}+\|\sqrt\r\dot u\|_{L^2}+\|\vr\|_{L^2}\big)\big(\|\nabla\dot u\|_{L^2}+\|\nabla u\|_{L^2}^2\big)\nonumber\\
&\le C\big(\|\nabla u\|_{L^2}+\|\sqrt\r\dot u\|_{L^2}+\|\vr\|_{L^2}\big)\|\na\dot u\|_{L^2}+C\big(\|\nabla u\|_{L^2}^2+\|\sqrt\r\dot u\|_{L^2}^2+\|\nabla u\|_{L^2}^4+\|\vr\|_{L^2}^2\big)\nonumber\\
&\le C(1+\|\na\dot u\|_{L^2})e^{-\eta_1t},
\end{align}
which along with \eqref{w41} implies that, for any $\delta>0$, there exists a positive constant $T_1$ such that
\begin{equation}\label{3.51}\int_{T_1}^\infty
\|F\|_{L^\infty}dt\le \delta.
\end{equation}

Let $X(t, y)$ be the particle path given by
\begin{align}\label{lz1}
\begin{cases}
\f{d}{dt}X(t, y)=u(t, X(t, y)),\\
X(0, y)=y,
\end{cases}
\end{align}
then we rewrite the mass conservation equation \eqref{a1}$_1$ as
\begin{align}\label{gg42}
\frac{\mathrm{d}}{\mathrm{d}t}\ln \rho=-\divv  u.
\end{align}
By the definition of $F$ in \eqref{f5}, \eqref{a9}, and \eqref{3.53}, we have
\begin{align*}
\rho(t, x)\ge \rho_0\text{e}^{{-\int_0^t}\|\divv  u\|_{L^\infty}\mathrm{d}\tau}\ge \rho_*\text{e}^{{-\int_0^t}C\|(P-\bar P,F)\|_{L^\infty}\mathrm{d}\tau}\ge \rho_*\text{e}^{-C(t+1)}.
\end{align*}

Now we claim that there exists a positive constant $c_0$ such that
\begin{equation}
\rho(t, x)\ge c_0.\label{3.50}
\end{equation}
If \eqref{3.50} were false, there exists a time $T_2\ge T_1$ such
that $0<c_0=\rho({ x}(T_2),T_2)\le \big(\frac{\bar\rho_0^\gamma}{2}\big)^\frac{1}{\gamma}$, then
choose a minimal value of $T_3>T_2$ such that $\rho({ x}(T_3),T_3)=\frac{c_0}{2}$.
Thus, $\rho({ x}(t),t)\in[\frac{c_0}{2},c_0]$ for $t\in[T_2,T_3]$.
one can rewrite $\eqref{a1}_1$ as
\begin{align}\label{3.52}
\f{d}{dt}\vr(t, X(t, y))+\frac{1}{2\mu+\lambda}\big(\rho(\r^\gamma-\bar\rho_0^\gamma)\big)(t, X(t, y))=
-\frac{1}{2\mu+\lambda}\r F(t, X(t, y)).
\end{align}
Integrating \eqref{3.52} along particle
trajectories from  $T_2$ to $T_3$,
abbreviating $\rho(t, X(t, y))$ by $\rho(t)$ for convenience, we get that
\begin{equation}\label{b3.53}
(2\mu+\lambda)(\rho-\bar\rho_0)\Big{|}^{T_3}_{T_2}+\int^{T_3}_{T_2}\rho(\rho^\gamma-\bar\rho_0^\gamma)\mathrm{d}\tau
=-\int^{T_3}_{T_2}\r F(t, X(\tau, y))\mathrm{d}\tau.
\end{equation}
We deal with the terms on the left-hand side of \eqref{b3.53} as follows
\begin{align}
(2\mu+\lambda)(\rho-\bar\rho_0)\Big{|}^{T_3}_{T_2}+\int^{T_3}_{T_2}\rho(\rho^\gamma-\bar\rho_0^\gamma)\mathrm{d}\tau
&\le -\frac{c_0(2\mu+\lambda)}{2}+\int^{T_3}_{T_2}\frac{c_0}{2}
\bigg(\frac{\bar\rho_0^\gamma}{2}-\bar\rho_0^\gamma\bigg)\mathrm{d}\tau\nonumber
\\&= -\frac{c_0(2\mu+\lambda)}{2}-\frac{c_0\bar\rho_0^\gamma}{4}(T_3-T_2).\label{3.54}
\end{align}
By \eqref{3.51}, we estimate the term on the right-hand side of \eqref{b3.53} as follows
\begin{align}\label{3.55}
-\int^{T_3}_{T_2}\r F(t, X(\tau, y))\mathrm{d}\tau\ge -c_0\delta.
\end{align}
Substituting \eqref{3.54} and \eqref{3.55} into \eqref{b3.53}, we obtain that
\begin{align*}
-\frac{c_1(2\mu+\lambda)}{2}-\frac{c_1\bar\rho_0^\gamma}{4}(T_3-T_2)\ge -c_0\delta,
\end{align*}
which is impossible if $\delta$ is small enough. The proof of \eqref{3.50} is finished.

Multiplying \eqref{3.52} by $\vr$ and using \eqref{3.50} and \eqref{w9}, we obtain that there exist two positive constants $\eta_7$ and $C$ such that
\begin{align}\label{w65}
\f{d}{dt}\vr^2(t, X(t, y))+\eta_7\vr^2(t, X(t, y))\le C\|F\|_{L^\infty}^2,
\end{align}
which together with \eqref{3.53} implies that
\begin{align}\label{w68}
\f{d}{dt}\vr^2(t, X(t, y))+\eta_7\vr^2(t, X(t, y))\le C(1+\|\na\dot u\|_{L^2})e^{-\eta_1t}.
\end{align}
Integrating \eqref{w68} along particle trajectories from $0$ to $t$, we derive from H\"older's inequality and \eqref{w41} that
\begin{align*}
\|\rho-\bar\rho_0\|_{L^\infty}^2&\le Ce^{-\eta_7t}+C\int_0^te^{-\eta_7(t-\tau)}(\|\nabla\dot u\|_{L^2}+1)e^{-\eta_1\tau}d\tau\nonumber\\
&\le Ce^{-\eta_7t}+C\int_0^\frac t2e^{-\eta_7(t-\tau)}(\|\nabla\dot u\|_{L^2}+1)e^{-\eta_1\tau}d\tau
+\int_{\frac t2}^te^{-\eta_7(t-\tau)}(\|\nabla\dot u\|_{L^2}+1)e^{-\eta_1\tau}d\tau\nonumber\\
&\le Ce^{-\eta_7t}+Ce^{-\frac{\eta_7t}{2}}\bigg(\int_0^\frac t2\|\nabla\dot u\|_{L^2}^2d\tau\bigg)^\frac12
\bigg(\int_0^\frac t2e^{-2\eta_1 \tau}d\tau\bigg)^\frac12
+Ce^{-\frac{\eta_7t}{2}}\int_0^\frac t2 e^{-\eta_1\tau}d\tau
\nonumber\\
&\quad+Ce^{-\frac{\eta_1t}{2}}\int_{\frac t2}^te^{-\eta_7(t-\tau)}d\tau
+Ce^{-\frac{\eta_1t}{2}}\bigg(\int_{\frac t2}^t\|\nabla\dot u\|_{L^2}^2d\tau\bigg)^\frac12
\bigg(\int_0^\frac t2e^{-2\eta_7(t-\tau)}d\tau\bigg)^\frac12\nonumber\\
&\le C\Big(e^{-\frac{\eta_7t}{2}}+e^{-\frac{\eta_1t}{2}}\Big),
\end{align*}
from which the conclusion \eqref{w61} follows.
\end{proof}

Now we are ready to prove Theorems \ref{thm1} and \ref{thm2}.
\begin{proof}[\textbf{Proof of Theorem \ref{thm1}}]
\eqref{1.8} follows from \eqref{w1}, \eqref{w17}, and \eqref{w42}, while \eqref{b1.11} follows from \eqref{w61}.
\end{proof}
\begin{proof}
[\textbf{Proof of Theorem \ref{thm2}}] Thanks to \eqref{a9}, \eqref{w1}, \eqref{w17}, \eqref{w41}, \eqref{w42}, and \eqref{3.53}, one gets from \eqref{gg42}
and $\inf_{x\in\Omega}\rho_0(x)=0$ that
\begin{align*}
\inf_{x\in\Omega}\rho(x, t)\le \inf_{x\in\Omega}\rho_0(x)e^{\int_0^t\|\div u\|_{L^\infty}d\tau}\le
\inf_{x\in\Omega}e^{C\int_0^t(\|F\|_{L^\infty}+\|P-\bar P\|_{L^\infty})d\tau}\le \inf_{x\in\Omega}\rho_0(x)e^{C(1+t)}=0,
\end{align*}
as the desired \eqref{aa11}.
\end{proof}

\section*{Conflict of Interest}
The authors have no conflicts to disclose.

\section*{Data availability}
No data was used for the research described in the article.


\end{document}